\documentclass{amsart}

\usepackage{amsfonts}
\usepackage{verbatim}
\usepackage{stmaryrd}
\usepackage{amsmath}
\usepackage[mathscr]{eucal}
\usepackage{color}

 \newtheorem{thm}{Theorem}[section]
 \newtheorem{cor}[thm]{Corollary}
 \newtheorem{lem}[thm]{Lemma}
 \newtheorem{prop}[thm]{Proposition}
 \theoremstyle{definition}
 \newtheorem{defi}[thm]{Definition}
 \theoremstyle{remark}
 \newtheorem{rem}[thm]{Remark}
 
 \numberwithin{equation}{section}

\newenvironment{pf}{\noindent{\it Proof.} }{\fbox{}\\}

\newcommand{\+}{\dag}

\renewcommand{\d}{\mathrm{d}}
\newcommand{\e}{\mathrm{e}}
\renewcommand{\i}{\mathrm{i}}

\renewcommand{\Im}{\mathrm{Im }}
\renewcommand{\r}{\mathfrak{r}}
\newcommand{\dc}{\overline{\partial}}
\title{Studies on the \\Spectral Schwartz Distribution}
\author{Wilhelm von Waldenfels\\
Universit\"at Heidelberg}
\date{}

\begin{document}

\begin{abstract}
 The resolvent function of an operator in a Banach space  is defined on an open subset of the complex plane and is holomorphic.  
 It obeys the resolvent equation. A generalization of this equation to Schwartz distributions is defined and a Schwartz distribution, which
 satisfies that equation is called a resolvent distribution. 
 In important cases the resolvent distribution is the whole plane.
 Its restriction to the subset, where it is continuous, is the usual resolvent function.
 Its complex conjugate derivative is ,but a factor, the spectral Schwartz distribution,
which is carried by a subset of the spectral set of the operator. 
 The spectral distribution yields a spectral decomposition. We have a generalized orthogonality relation.
 Completeness is defined in a natural way and is the case e.g. if the operator is bounded.
The spectral distribution of a matrix and a unitary operator are given. If the the 
operator is a self-adjoint operator on a Hilbert space, the spectral distribution is the derivative of the spectral family. We calculate the
spectral distribution  and the eigen value problem of the multiplication operator and  of the multiplication operator with
a non symmetric rank one perturbation. The operator  is not normal
and may have discrete real or imaginary eigenvalues or a nontrivial Jordan decomposition.
\end{abstract}

\maketitle
{\bf Keywords}:  Schwartz distributions, Distribution Kernels, Resolvent,\\ Spectral Decomposition, Spectral Distribution, Multiplication Operator.46A20 Distributions as Boundary Values of Analytic Functions, \\46A24 Distributions on infinite dimensional spaces.
\section{Definition and Basic Properties}
\subsection{Introduction}In a study on radiation transfer Garyi V. Efimov, Rainer Wehrse  and the author 
\cite{EvW} developed a method, how to calculate the
spectral decomposition of a  multiplication operator in $\mathbb{R}^n$ perturbed by a   finite rank operator.
This method was later applied in four examples coming from the theory of quantum stochastic processes \cite{vW}. 
In these cases the operator was defined on a Hilbert space of functions on the $\mathbb{R}^n$ and the resolvent
$R(z)$ had the property, that the scalar function $z \mapsto \langle f_1| R(z)| f_2  \rangle$, which is defined for
$z$ in the resolvent set, can be extended to the whole plane, if $f_1,f_2$ are $C^\infty$ functions with compact support.
The function $\langle f_1| R(z)| f_2  \rangle $  is holomorphic on the resolvent set. Hence its complex derivative
$ \dc \langle f_1| R(z)| f_2  \rangle$  is supported by the spectral set. In  \cite{vW} the spectral Schwartz distribution $M(z)$ was defined scalarly by
$$  \langle f_1| M(z)| f_2  \rangle = (1/\pi) \dc\langle f_1| R(z)| f_2  \rangle$$
From the so defined spectral distribution one could deduce the spectral decomposition namely the proper and the generalized eigen vectors. We proved by hand the orthogonality relations and the completeness.  Our examples were all self adjoint operators in a Hilbert space or in Krein space \cite{Langer}. We did however not use the corresponding theories but used
directly the theory of Schwartz distributions how it is explained in the books of L.Schwartz \cite {Schwartz1},\cite{Schwartz4} and in the books of Gelfand
and Vilenkin \cite{Gelfand1},\cite{Gelfand2} . We had the advantage that we did not  only get the existence of the spectral decomposition but the explicit form of
the eigen vectors.

Our treatment was unsatisfactory in the sense that we had to prove orthogonality and completeness by hand.  Therefor
we introduced
in this paper  a stronger version of the spectral Schwartz distribution. The spectral distribution is not only scalarly defined but an operator valued distribution.

The resolvent function obeys on the resolvent set
the {\em resolvent equation}. 
 A consequence of the resolvent equation is that the resolvent function is holomorhic on the resolvent set.
The resolvent equation can be generalized for Schwartz   distributions in a natural way. This equation  is called the {\em resolvent distribution equation}. It might happen, that the resolvent function can be extended to a Schwartz distribution on the whole plane which satisfies the resolvent distribution equation. We denote this distribution again by $R(z)$. It is holomorphic on the resolvent set.   
Therefor the conjugate complex derivative of $R(z)$ vanishes on the resolvent set and is supported by the spectral set.
We define the spectral Schwartz distribution by
$$ M(z)= (1/\pi) \dc R(z) .$$
The spectral distribution is a distribution with values  in a Banach algebra and has much stronger analytical properties than the only scalarly defined distribution above.
 It is a generalized eigen projector, 
 we have the usual orthogonality relations and we have completeness in many cases, especially if the resolvent function 
is the resolvent of a bounded operator. In example 5 we have a typical problem which was treated formerly by the scalar spectral distribution and is now treated by the operator valued spectral distribution. 

The paper is the correction and the improvement of my article in the arXiv \cite{vW1} and coincides to a large extent with that.

\subsection{Schwartz distributions}
We use in this paper mostly Schwartz distributions with values in a Banach space provided with the norm topology.
We recall the definition of a distribution given by L.Schwartz   \cite{Schwartz1}. A distribution on  an open set
 $G \subset\mathbb{R}^n$ is a linear functional $T$
 on the space of {\em test functions} $\mathcal{D}(G)$,
that is the space
of all real or complex valued  infinitely differentiable functions of compact support (  or $C_c^\infty$-functions) with support in $G$.   The functional $T$ 
mapping $\mathcal{D}(G)$ into a Banach space is continuous in the following way: if the functions $\varphi_n\in \mathcal{D}(G)$ have their support
in a common compact set and if they  and all their derivatives converge to 0 uniformly,
 then the $T(\varphi_n)\to 0$ in norm.The theory of distributions  with values in a Banach space or in a topological vector space is not essentially different from the theory of real distributions as it is pointed out by L.Schwartz \cite{Schwartz1}p.30. We use the results of the theory of real or complex  distributions without mentioning.
 
   A set $B \subset \mathcal{D}(G)$ is bounded, if there
exists a compact set $K$ in $G$ and sequence $(m_0,m_1,m_2,\cdots)$ of numbers $> 0$ such that all $\varphi \in B$ have their support 
in $K$ and their derivatives of order $\le k$ are bounded by $m_k$. If a sequence $T_j$ converges for any $\varphi$,
then it converges uniformly on bounded sets in $\mathcal{D}(G)$.

We write often
$$ T(\varphi)=\int T(x) \varphi(x) \d x.$$
The integral notation is the usual notation of physicists and  has been adopted by Schwartz with minor modifications in his articles
 on vector-valued distributions \cite{Schwartz2}\cite{Schwartz3}. It emphasizes the fact, that a distribution can be considered
as a {\em generalized function}. This is the notation in the Russian literature \cite{Gelfand1}.  Similar is
 de Rham's formulation \cite{Rham} : $T$ can be considered as a current of degree 0 applied to a form
of degree $n$ namely
$ \varphi(x_1,\cdots,x_n) \d x_1 \cdots \d x _n$. We use variable transforms of distributions accordingly.

We have often to deal with the so called
$\delta$-function given by $\delta(\varphi)=\varphi(0)$ and the generalized function
$ \mathcal{P}/x = (\d/\d x)\ln|x|$, where $\mathcal{P}$ stands for principal value. One has for $u \downarrow 0$
\begin{equation}
\frac{1}{x + \i u}= \frac{x}{x^2+u^2}-\frac{\i u}{x^2+u^2}\to \frac{1}{x+\i 0}= \frac{\mathcal{P}}{x} - \i \pi \delta(x) 
\end{equation}

In order to include  Dirac's original ideas, L.Schwartz introduced the {\em distribution kernels} \cite{Schwartz2}.
A distribution kernel $K(x,y)$ on $\mathbb{R}^m \times \mathbb{R}^n$ is a distribution on that space.
 If $T$ is a distribution on $\mathbb{R}^n$ one may associate to it a kernel $T(x-y)$ on
$\mathbb{R}^n \times \mathbb{R}^n$  by
$$ \iint T(x-y) \varphi(x) \psi(y) \d x \d y = \int \d x T(x)\varphi(x)\int \d y \psi(x-y)  ,$$
 We have
$$ \iint \delta(x-y) \varphi(x,y) \d x \d y = \int \d x \varphi(x,x).$$So
$ \delta(x-y)=\delta(y-x)$.

If $S(x,y)$ and $T(y,z)$ are two kernels one may form a distribution of three variables 
$S(x,y)T(y,z)$, if this exists. If the kernels are $S(x-y)$ and $T(y-z)$ then $S(x-y)T(y-z)$ exists and  is given by
\begin{multline*} \iiint \d x \d y \d z S(x-y)T(y-z) \varphi(x,y,z)\\= \iiint \d u \d v \d y S(u)T(v) \varphi(u+y,y,y-v).\end{multline*}

We obtain for any distribution $T$ the relation
\begin{equation}\delta(x-y)T(y-z)= \delta(x-y)T(x-z). \end{equation}
At the end of the paper we shall need th following lemma
\begin{lem}
We have the formulae 
\begin{align*} &\iiint \delta(x-y)\delta(y-z) \varphi(x,y,z)\d x \d y \d z=  \int \varphi(x,x,x) \d x \\&
\frac{\mathcal{P}}{x-\omega}\frac{\mathcal{P}}{y-\omega}=
\frac{\mathcal{P}}{y-x}\left(\frac{\mathcal{P}}{x-\omega} -\frac{\mathcal{P}}{y-\omega}
\right) + \pi^2 \delta(x-\omega)\delta(y-\omega) \end{align*}
\end{lem}
\begin{pf}
The first equation is trivial, the second one is contained in \cite[p.74]{vW}. We present a somewhat easier proof. 
\begin{multline*}\frac{1}{x-\omega+\i u}\frac{1}{y-\omega+\i u}= \frac{1}{y-x}(\frac{1}{x-\omega+\i u}-\frac{1}{y-\omega+ \i u})\\
=\frac{\mathcal{P}}{y-x}(\frac{1}{x-\omega+\i u}-\frac{1}{y-\omega+\i u})\end{multline*}
as the factor of $\mathcal{P}/(y-x)$ vanishes for $y=x$. Go with $u \downarrow 0$ and obtain with the help of (1.1) and (1.2)
for the left hand side
\begin{multline*}\big (\frac{\mathcal{P}}{x-\omega}- \i \pi \delta(x-\omega)
\big)\big (\frac{\mathcal{P}}{y-\omega}- \i \pi \delta(y-\omega)\big)\\
=\frac{\mathcal{P}}{x-\omega}\frac{\mathcal{P}}{y-\omega}
- \pi^2 \delta(x-\omega)\delta(y-\omega) - \i \pi \delta(x-\omega) \frac{\mathcal{P}}{y-\omega}
-\i \pi \delta(y-\omega)\frac{\mathcal{P}}{x-\omega} \\
 =\frac{\mathcal{P}}{x-\omega}\frac{\mathcal{P}}{y-\omega}
- \pi^2 \delta(x-\omega)\delta(y-\omega) - \i \pi \delta(x-\omega) \frac{\mathcal{P}}{y-x}
-\i \pi \delta(y-\omega)\frac{\mathcal{P}}{x-y}\end{multline*}
For the right hand side we get
$$\frac{\mathcal{P}}{y-x}\bigg(\frac{\mathcal{P}}{x-\omega}-  \i \pi \delta(x-\omega)-
\frac{\mathcal{P}}{y-\omega}+ \i \pi \delta(y-\omega)\bigg)$$
Compare with the left hand side.
\end{pf}

\subsection{Notion of the resolvent function}
We define at first the notion of a resolvent function, as it is given by Hille and Phillipps \cite{Hille} under the name of pseudo-resolvent. 
Assume an  an open set $G\subset\mathbb{C}$ and a Banach algebra $B$ and
 a function $R(z):G\to B$ satisfying the {\em resolvent equation}
  \begin{equation}R(z_1)-R(z_2)=(z_2-z_1)R(z_1)R(z_2) \end{equation}
  then we call $R(z)$ a {\em resolvent function} on $G$. This equation implies, that $ R(z)$ is holomorphic in $G$, i.e.
  that $R(z)$ can be locally expanded into a power series
$R(z)=\sum_{n \ge 0} a_n (z-z_0)^n$ with $\sum |z-z_0|^n \|a_n\| < \infty$ for any $z_0\in G$. Hence the equation
makes sense
\begin{equation}
 R(z_1)R(z_2)= 1/(z_2-z_1)(R(z_1)-R(z_2)
\end{equation}

 Assume a Banach space $V$ and the Banach algebra $B=L(V)$ of operators $B \to B$ provided with the operator norm and assume a closed operator $A$ defined in $V$ with values in $V$. The subspace
 $\rho(A)\subset \mathbb{C}$, where $R(z)= (z-A)^{-1}$ exists,is called the {\em  resolvent set } and  $R(z)$ 
 the {\em resolvent function} of $A$. The set $\rho(A)$ is open and $ R(z): \rho(A)\to L(V)$ satisfies the resolvent equation.
  
Go the other way round and  assume  an open set $G\subset\mathbb{C}$ and
 a function $R(z):G\to L(V)$ satisfying the resolvent equation.
The function  $R(z)$
      is \emph {holomorphic} in $G$.
The subspace $ D= R(z)V$ is a subset independent of $z\in G$. If
$R(z_0)$ is injective for one $z_0\in G$, then $R(z)$ is injective for all $z\in
G$ and there exists a mapping $A:D\to V$ such that
 \begin{align*} (z-A)R(z)f&=f
\text{ for } f\in V & R(z)(z-A)f&=f \text{ for } f\in D \end{align*}

  The operator $A$ is closed and $R(z)$ is the resolvent of $A$\cite{Hille}.
 There exists not always such an operator corresponding to a resolvent , e.g. the function $R(z)=0$ for all $z$ fulfills the 
 resolvent equation and $R(z)$ is surely not injective.

\subsection{Definition of the resolvent distribution} 
The equation (1.4) can easily be generalized to Schwartz distribution using the fact, that $z \mapsto 1/z$ is locally integrable.
\begin{defi} Assume a Schwartz distribution $\varphi \mapsto	R(\varphi)$ on an  open set $G\subset \mathbb{C}$
, which satisfies the equation
\begin{multline*}R(\varphi_1)R(\varphi_2)= \int \d^2 z_1R(z_1) \varphi_1(z_1)
\int \d^2 z_2 \varphi_2 (z_2) /(z_2-z_1)\\
- \int \d^2 z_2R(z_2) \varphi_2(z_2)
\int \d^2 z_1 \varphi_1 (z_1) /(z_2-z_1)\end{multline*}
then we say, that $R(\varphi)$ satisfies the {\em distribution resolvent equation} and call $R(\varphi)$ a {\em  resolvent distribution}.
Here $\d^2 z = \d x \d y$ is the surface elament on the complex plane for $z=x + \i y$.
\end{defi}

 If $R(z)$ obeys the resolvent equation in $G$
 it might be, that there is Schwartz distribution on  an open set $G' \supset G$
extending $R(z)$ and obeying
 the distribution resolvent equation.
Assume a distribution $R(z)$ on an open set $G'$ satisfying the distribution resolvent equation , whose restriction
on  an open subset $G \subset G'$ equals a continuous function, then the two integrals on the right hand side 
are usual integrals over locally integrable functions on $G$ and hence $R(z)$ 
satisfies the usual resolvent equation on $G$ and is holomorphic.

\subsection{Complex conjugate derivation}

    If $G\subset \mathbb{C}$ is open and $f:G\to \mathbb{C},f(z)=f(x+\i y)$ has
a continuous derivative, set
                 \begin{align} \partial f &=\frac{\d f}{\d z} =
                 \frac{1}{2}\left(\frac{\partial f}{\partial x}-  \i \frac
{\partial f}{\partial y}\right)&
                 \overline{\partial} f& =\frac{\d f}{\d
\overline{z}}
     =
                 \frac{1}{2}\left(\frac{\partial f}{\partial x}+  \i
\frac {\partial
  f}{\partial y}\right).\end{align}
    The operator $\partial$ is the usual derivative, $\overline{\partial}$ is called the {\em complex conjugate derivative}.
     The function $f$ is
 holomorphic if and
       only if $\overline{\partial}f =0$.
       
         In an analogous way
  one defines these derivatives for
                 Schwartz distributions. If $T$ is a distribution on $\mathbb{C}$, then $\overline{\partial} T(\varphi)=-T(\overline{\partial}\varphi)$. If $T$ is  distribution in an open set and $\dc T = 0$,  then $T$ is a holomorphic function in the usual sense \cite {Schwartz4},p.72.

                 We need the following variant of Gauss's theorem

   \begin{lem}Assume $G\subset \mathbb{C}$ an open subset , $f:G\to \mathbb{C}$ continuously differentiable and $G_0\subset G$
  an open subset with compact closure and smooth border $G_0'$, 
  then
  $$ \int_{G_0} \overline{\partial} f(z)\d^2z= -\frac{\i}{2}\int_{G_0'}f(z) \d z$$
   Here $\d z$ denotes the line element directed in the sense, that the interior of $G_0$ is on the left hand side\end{lem}    
   We cite three classical results \cite{Schwartz1}
   p.38f  
   \begin{prop}         
 The
  function $z\mapsto 1/z$ is
                                     locally integrable and one has
\begin{equation}
   \overline{\partial}(1/z)=\pi\delta(z).\end{equation}
   The function $z \mapsto 1/z^{n+1}$ is not integrable in the origin for $n>0$ One defines the distribution
   $$ \mathcal{P}/z^{n+1}=(-1)^n / n! \partial^n(1/z). $$
   One obtains 
   $$ \int \d ^2 z (\mathcal{P}/ z^{n+1}) \varphi (z)= \lim_{\varepsilon \downarrow 0}
   \int_{|z| \ge \varepsilon} (1/z^{n+1}) \varphi(z)$$
   and
    \begin{equation}
\dc (\mathcal{P}/z^{n+1})=\frac{(-1)^n \pi} { n!} \partial^n \delta(z).
\end{equation}
                 Assume $f$ to be defined
and
     holomorphic for the elements $x+\i y \in G, y\ne 0$, and that $f(x+\pm \i 0)$
 exists and is continuous, then 
              \begin{equation}\overline{\partial}f(x+\i y)= (\i/2)(f(x+\i 0)-f(x-\i 0))\delta(y)\end{equation}
              The last equation holds for operator valued functions as well.
              \end{prop}
 \begin{pf}             
We prove equation (1.6). 
  \begin{multline*} \int \d ^2 z\dc (1/z) \varphi(z)  = -  \int \d^2 z (1/z) \dc \varphi(z)
  =    - \lim _{\varepsilon \downarrow 0}  \int_{|z| \ge \varepsilon} \d^2 z (1/z) \dc \varphi(z)  
\\=   - \lim _{\varepsilon \downarrow 0}  \int_{|z| \ge \varepsilon} \d^2 z  \dc((1/z)  \varphi(z))     
= - \lim _{\varepsilon \downarrow 0}  \int_{C(\varepsilon)} \d z (1/z)  \varphi(z) 
 \to \pi \varphi(0)\end{multline*}   
 where $C(\varepsilon)$ is the circle around the origin run in the counter clockwise sense. 
 Equation (1.8) is shown in a smilar way.     In order to prove (1.7) one has to expand $\varphi$ into a Taylor series.    
\end{pf}

\subsection{Definition of the spectral Schwartz distribution and basic results}
\begin{defi} Assume $R(z)$ satisfying the distribution resolvent equation, then we call
 $$ M(z)=(1/\pi)\overline{\partial}R(z)$$
the {\em spectral Schwartz distribution} of $R(z)$.
\end{defi}

If the resolvent distribution $R(z)$ is in an open set  the extension of
 the resolvent  function of an operator $A$, we call $M(z)$  a spectral distribution of $A$. Remark,
 that there may be many spectral distributions of $A$.  
  As $R(z)$ is holomorphic on an open subset, where it is continuous,   $M(z)$ is supported by the spectral set of $R(z)$.    
  
  Using the abbreviation $\r(z)=1/z$ we may write the the distribution resolvent equation  in the form
  \begin{equation} R(\varphi_1)R(\varphi_2)=-R((\r*\varphi_1)\varphi_2+\varphi_1(\r*\varphi_2))\end{equation}  
  where the $*$ denotes the convolution.
  The following theorem shows, that the spectral distribution can be considered as the generalization of the family of eigen projectors
  in the finite dimensional case. The operator $M(z)$ corresponds to an eigen projector of the eigen value $z$.
   We have a generalized orthogonality relation and the fact, that the product with the resolvent
  means inserting the corresponding eigenvalue.
  
  \begin{thm} The spectral distribution is multiplicative. So assume  two $C_c^\infty$ functions $\varphi_1,\varphi_2$,
  then
    \begin{align*}&M(\varphi_1)M(\varphi_2)= M(\varphi_1\varphi_2)&\text{or}&\quad
             M(z_1)M(z_2)=\delta(z_1-z_2)M(z_1).\end{align*}
            Conversely
 assume a distribution $M(z)$  in $ G \subset \mathbb{C}$  obeying the equation\\
            $ M(z_1)M(z_2)=\delta(z_1-z_2)M(z_1).$
             Assume that the integral 
             \begin{align*}&
             R(z)= \int \d^2 \zeta  M(\zeta)/(z-\zeta)= (M*\mathfrak{r})(z)\\&
            \int R(z) \varphi(z)\d^2 z= \int M(\zeta)\d^2 \zeta \int \varphi(z) /(z-\zeta) \d^2 z \end{align*}
            exists for $\varphi\in C_c^\infty$, then $R(z)$ fulfills the resolvent equation for distributions.
            Furthermore
          $$    M(z_1)R(z_2)=  M(z_1) /(z_2-z_1), $$
       \begin{multline*}  \int M(z_1)\varphi_1(z_1)\d ^2z_1\int R(z_2)\varphi_2(z_2) \d ^2z_2\\
= \int M(z_1) \varphi_1\d ^2 z_1\int \varphi_2(z_2)/(z_2-z_1)\d^2 z_2
\end{multline*}
            \end{thm}
            \begin{pf}
            Use the equation
            $ \overline{\partial} \r = \pi \delta$
and hence
$$ \r*\dc \varphi= \r*\dc	\delta * \varphi = \dc \delta * \r * \varphi =\dc\r*\varphi= \pi \delta *\varphi = \pi \varphi$$ and
 calculate
\begin{multline*}M(\varphi_1)M(\varphi_2) = \pi^{-2} R(\dc\varphi_1 )R(\dc \varphi_2)= -\pi^{-2}R((\r*\dc \varphi_1)\dc\varphi_2
+\dc \varphi_1(\r*\dc \varphi_2))\\ 
= -\pi	R(\varphi_1 \dc \varphi_2 +( \dc \varphi_1)\varphi_2)=-\pi R(\dc(\varphi_1\varphi_2))	=M(\varphi_1 \varphi_2).\end{multline*}
In order to prove the converse assertion, observe that it means that $\r*M$ is a resolvent distribution. Now
$M(\varphi)=-1/\pi R(\dc \varphi)$ and $(\r*M)(\varphi)=-M(\r*\varphi)$. Hence
   \begin{multline*}R(\varphi_1)R(\varphi_2)=M(\r*\varphi_1)M(\r*\varphi_2)= 
   M((\r*\varphi_1)(\r*\varphi_2))\\
   =-1/\pi R(\dc((\r*\varphi_1)(\r*\varphi_2))=-R((\r*\varphi_1)\varphi_2+\varphi_1(\r*\varphi_2))\\
\end{multline*}
The last equation says 
$M(\varphi_1)R(\varphi_2) =-M(\varphi_1(\r*\varphi_2))$.
In fact
\begin{multline*}M(\varphi_1)R(\varphi_2)  = -\dfrac{1}{\pi} R(\dc \varphi_1)R(\varphi_2)=
\dfrac{1}{\pi	} R((\r*\dc \varphi_1)\varphi_2 + \dc \varphi_1 (\r*\varphi_2))\\=
\dfrac{1}{\pi	} R(\pi \varphi_1\varphi_2 + \dc \varphi_1 (\r*\varphi_2))
=\dfrac{1}{\pi}R(\dc (\varphi_1(\r*\varphi_2))=-M(\varphi_1(\r*\varphi_2)) .\end{multline*}
\end{pf}
The operator $M(z)$ is a generalized eigen projector. In fact if $R(z)=1/(z-A)$ then\\
$$M(z_1)(1/(z_2-A)) =  M(z_1) /(z_2-z_1)$$
and the relation
$$ AR(z)=R(z)A=-1 +zR(z)$$
yields
$$ AM(z)=M(z)A= zM(z)$$
as $\dc z =0$.
 The relation $M(z_1)M(z_2)=\delta(z_1-z_2)M(z_1)$ is a generalized orthonormality relation.
 
   \begin{prop} If $R$ is the resolvent of a bounded operator $A$ and  can be extended to a a resolvent distribution  defined
   on the whole complex plane and called $R$ again, then
   $$AR(z) = R(z) A= -1 + zR(z)$$ 
 $$ AR(\varphi)= R(\varphi)A=\int \d ^2 z \varphi(z)(-1 + zR(z)) $$
\end{prop}
\begin{pf}
 If $\zeta$ is in the resolvent set of $A$, then resolvent distribution equation yields
 $$ R(\zeta)R(\varphi)= \int \ d ^2 z \frac{1}{\zeta-z}(R(z)-R(\zeta) ) \varphi(z)$$
 Expand in powers of $1/\zeta$ for $\zeta$ sufficiently big
 $$ (1/\zeta +A/ \zeta^2 +\cdots)R(z)	= \int \d ^2 z (1/\zeta +z/ \zeta^2 +\cdots)
 (R(z)- (1/\zeta +A/ \zeta^2 +\cdots)\varphi(z)$$
 and compair the powers of $1/\zeta^2$.
 \end{pf}
 The following theorem shows, that the spectral distribution is complete if it belongs to a bounded operator.
    \begin{thm}
   If $R$ is the resolvent of a bounded operator and  can be extended to a a resolvent distribution defined
   on the whole complex plane and called $R$ again , then the support of $M$ is compact  and $M(1)$ is defined , where here $1$ is the constant function $1$,
   and
   $$M(1)=\int \d^2 z M(z) =1_{L(V)}$$
   In this case we call $M$ {\em complete}.
   \end{thm}
   \begin{pf}Assume, that the support of $M$ is contained in the circle of radius $r$. Assume a test 
    function $\varphi$	constant $1$ on this
   circle, then 
   \begin{multline*}
   $$ M(1)=M(\varphi)= -\frac{1}{\pi}\int  \d^2 z R(z) \dc \varphi(z)= -\frac{1}{\pi}\int _{|z|>r} \d^2 z R(z) \dc \varphi(z)=\\
    -\frac{1}{\pi}\int_{|z|>r} \d^2 z \dc( R(z)  \varphi(z))=
   \frac{1}{2\pi\i}\int_\Gamma \d z R(z)= 1_{L(V)},\end{multline*}
   where $\Gamma$ is the circle of radius $r$ run in the anti-clockwise sense.
    Now the residuum  in infinity of $R(z)=1/(z-A)$ equals $1_{L(V)}$.
   \end{pf}

\begin{thm}
Assume an open set $G \subset \mathbb{C}$ and  and a closed set $G_0 \subset G $ of Lebesgue measure 0.
Assume a  resolvent function on $G \setminus G_0$ and a sequence of closed set $ G \supset  G_n \downarrow G_0$. 
Assume that for $\varphi \in \mathcal{D}(G)$
$$\lim_{n \to \infty} \int_{G\ \setminus G_n} \d ^2 z \varphi(z) R(z) = R(\varphi)$$ 
exists. then the distribution $R(\varphi)$ fulfills the resolvent distribution equation.
\end{thm}
\begin{proof}
We have
$$ R(\varphi_1)R(\varphi_2) = \lim _{m,n \to \infty} \int_{G \setminus G_m}\d^2 z_1  \int _{G\setminus G_n} \d^2 z_2
R(z_1)R(z_2)\varphi_1(z_1) \varphi_2(z_2)$$
Using the resolvent equation for functions we have to consider

$$ -\int_{G \setminus G_m}\d^2 z_1  \int _{G\setminus G_n} \d^2 z_2(\frac{R(z_1)}{z_2-z_1}+\frac{ R(z_2}{z_1-z_2})
\varphi_1(z_1)\varphi(z_2)
= I+II$$
with
\begin{multline*} I = - \int_{G \setminus G_m}\d^2 z_1  \int _{G\setminus G_n} \d^2 z_2\frac{ R(z_1)}{z_2-z_1}
\varphi_1(z_1)\varphi_2(z_2) \\=
\int_{G \setminus G_m}\d^2 z_1  R(z_1) \psi_n(z_1)\end{multline*}
and
$$ \psi_n(z_1)= \varphi_1(z_1)\ \int _{G\setminus G_n} \d^2 z_2\varphi(z_2) /(z_1-z_2)$$
The support of $\psi_n$ is contained in the support of $ \varphi_1$. The derivatives of $\psi_n$ are linear cmbinations of the products of the derivatives of $\varphi_1$ and derivatives of the convolution 
$\ \int _{G\setminus G_n} \d^2 z_2\varphi(z_2) /(z_1-z_2)$, which are the convolutions of the derivatives s of $\varphi_2$. Hence $ \psi_n \in \mathcal{D}(G)$ and as $G_0$ is a Lebesgue
null set, 
$$ \psi_n(z) \to \psi(z) = \varphi_1(z_1)\ \int _G \d^2 z_2\varphi(z_2) /(z_1-z_2)$$
in the topology of $\mathcal{D}(G)$. The set $\{\psi,\psi_n, n=1,2,\cdots\}$ is a bounded set in  $\mathcal{D}(G)$
and the convergence 
$$
\int_{G \setminus G_m}\d^2 z_1  R(z_1) \psi_n(z_1)\to 
\int_G\d^2 z_1  R(z_1) \psi_n(z_1)$$
is uniform on this bonded set, i.e. uniform in $n$..From there one deduces that for $m,n \to \infty$
\begin{multline*}I = \int_{G \setminus G_m}\d^2 z_1  R(z_1) \psi_n(z_1)\\\to R(\psi)=- \int_{G }\d^2 z_1  \int _{G} \d^2 z_2R(z_1)/(z_2-z_1)\varphi_1(z_1)\varphi_2(z_2)\end{multline*}
A similar assertion holds for $II$.
\end{proof}

\subsection{ Singularities of the resolvent function in a point }
 \begin{prop}
  If the resolvent function $R(z)$  has a  pole of order $n$  in a point $z_0$, then  there exists  an open neighborhood $U$ of $z_0$ such that for $z\in U, z \ne z_0$the function  $R(z)$ is holomorphic  and can be expanded in a Laurent series
   $$ R(z)= b_0( z-z_0)^{-1} + \cdots b_{n-1} (z-z_0)^{-n} + R_0(z),$$
   where $R_0$ is a power series in $( z-z_0)$ with exponents $ \ge 0$.We have $b_0=p$ and $b_j= a^jj \ge 1$, where $p^2=p$ is idempotent and
 $a^n=0$ and $ap=pa$. Define
 $$ R(z) = \sum_{k=0}^{n-1} b_k\frac{\mathcal{P}}{(z-z_0)^{k+1}}+R_0(z),$$
 where the distribution \cite{Schwartz1}
 $$ \int \d^2 z \frac{\mathcal{P}}{z^k}\varphi(z)=\lim_{\varepsilon \downarrow 0} \int_{|z|> \varepsilon} \frac{1}{z^k}\varphi(z)\d^2 z
 = (-1)^k \int \frac{1}{z}\partial ^k \varphi(z)\d^2 z.$$
 The distribution $R(z)$ fulfills the resolvent distribution equation. The spectral distribution is
 $$ M(z)=  \sum_{k=0}^{n-1} \frac{(-1)^k}{k!} p a^k \partial^k \delta(z-z_0)$$
 with $pa^0=p$.
 \end{prop}
 \begin{pf}
 We write the resolvent equation in the form
 $$ \frac{R(z_1)}{z_1}\frac{R(z_2)}{z_2} =\frac{1}{z_1^{-1}-z_2^{-1}}(R(z_1)-R(z_2))$$ $$
 = \sum_{m=0}^{n-1} b_m (\sum_{k+l=m}z_1^{-k+1}z_2^{-l+1}) +\frac{z_1z_2}{z_2-z_1}(R_0(z_1)-R_0(z_2))$$
 and compare the coefficients of $z_1^{-k}z_2^{-l}$.  Then
 $$ b_k b_l = \begin{cases}
b_{k+l} & \text{ for } k+l \le n-1 \\
0 & \text{ for } k+l  > n-1 
\end{cases}.$$

From there follows directly the  assertion of the proposition.
Using theorem 1.9, we see that $R(z)$ fulfills the resolvent distribution equation. That $M(z)$ is multiplicative
 follows directly from Leibniz's formula. 
 \end{pf}

\begin{prop} 
 If the resolvent function $R(z)$ has two isolated poles of finite order $w_1,w_2$, then there exist two open neighborhoods 
 $U_1,U_2$ of $w_1$ and $w_2$, such that $R(z)$ is holomorphic in $U_1 \cup U_2$ outside the points $w_1,w_2$ and 
 $$ R(z) = \sum_k b_k (z-w_1)^{-k} + O(1)$$
 $$ R(z) = \sum_l c_k (z-w_2)^{-l} + O(1)$$ 
 near $w_1$, resp $w_2$. Then
$$ b_kc_l = 0 $$ 
\end{prop}
\begin{pf}
With the help of prposition 1.10 we can define a distribution on $U_1$ and second one on $U_2$. 
By  gluing together \cite{Schwartz1}p.26 these distributions can be extended to a distribution on $U_1 \cup U_2$ and by theorem 1.9 it fulfills the distribution resolvent equation extending $R$ and called $R$ again. Choose two
test functions $\varphi_1,\varphi_2$ with support in $U_1$ resp. $U_2$, with disjoint supports, then
$$M(\varphi_1)= \sum_k (1/k!)b_k \partial^k \varphi_1 (w_1) $$
$$M(\varphi_2)= \sum_l (1/l!)c_l \partial^l \varphi _2(w_2) $$
 $$0= M(\varphi_1\varphi_2)=   \sum_{k,l} (1/k!)b_k \partial^k \varphi_1 (w_1)(1/l!)c_l \partial^l \varphi _2(w_2)$$
 As $\partial^k \varphi_1 (w_1),\partial^l \varphi _2(w_2)$ can be chosen arbitrarily, the equation $ b_kc_l = 0$
 follows.
\end{pf}

\subsection{ Singularities of the resolvent function  on the real line}

       For the following discussion we  want to extend Dirac's bra and ket notation to functionals. Let $V$ be a  pre-hilbert space, i.e. complex vector space with strictly positive definite 
  scalar product
$ f,g \mapsto \langle f | g \rangle .$
  The scalar product is linear in the second factor and semi linear or conjugate linear in the first one.  
  We denote by $f = |f \rangle $ (ket vector) the elements of $V$ and by $\langle f|$ (bra vector)  the linear functional
   $ g \mapsto \langle f | g \rangle $. In the same spirit  $|f \rangle$ can be considered as semi linear (= conjugate complex linear) functional on $V$. If $\alpha$ is a linear functional on $V$, then we extend the notation and write
   $\alpha(f)= \langle \alpha| f \rangle$  and $\alpha = \langle \alpha |$		 and if $\alpha$ is a semi linear functional we write
   $\alpha(f)= \langle f | \alpha \rangle$ and $ \alpha = | \alpha \rangle $. We denote by $V^\+$ the space of all semi linear  functionals.  We embed $ V\hookrightarrow V^\+$.  If $V$ is a Hilbert space, then all semi linear functional are of the form $|\alpha \rangle = |f \rangle, f\in V$ and we can identify
   $V= V^\+$.

  We investigate now a resolvent functions with singularities on the real line. 
  Assume an open subset $G \subset \mathbb{R}$	and an open set  $G_1$ with $ G \subset \overline{G}\subset G_1$ and  an open interval $I$  containing $ 0$
,   and a resolvent function $R(z)$
 holomorphic in $(G_1\times I)\setminus G $. Define for $u\in I,u \ne 0$ the functions  $R_u(x)=R(x+\i u)$ and
  $\varphi_u(x)=\varphi(x+ \i u)$ on $G_1$.  Assume, that for $u\to +0$ and $u\to -0$ the function $R_u$
converges in the sense of distributions to $R_{+0}$ resp. $R_{-0}$. 
  For $\varepsilon > 0,[-\varepsilon,\varepsilon] \subset I$ the set
 $ \{ \varphi_u, |u| \le \varepsilon \}$ is bounded in $\mathcal{D}(G_1)$ and hence the convergence $R_u(\varphi)\to R_{\pm 0} (\varphi)$ is uniform on this bounded set and $R_u(\varphi_u)\to R_{\pm 0} (\varphi_0)$. So
 $ u \mapsto R_u(\varphi_u)$ is integrable and defines
   the distribution  $R(\varphi)$ on $ G_1\times I$
 $$ R(\varphi)= \int \d u  R_u(\varphi_u) = \int \d u \int \d x R(x+ \i u) \varphi(x+\i u).$$
 The integral defines a {\em resolvent distribution on $G_1\times I$} by theorem 1.9.We have
 $$ M(x + \i y) =\frac{1}{2 \pi \i }( R(x - \i 0)-R(x + \i 0))\delta(y)= \mu(x) \delta(y).$$
 So
  $$ \mu(x)=\frac{1}{2 \pi \i }( R(x - \i 0)-R(x + \i 0))$$
  is a distribution on $G_1$ and
  as a consequence of theorem 1.6 we have that $\mu$ is multiplicative and	for $ \zeta \notin G$
 $$ M(\varphi) R(\zeta) = \int \d^2 z\varphi(z)M(z)/(\zeta-z)= \int \d x \varphi(x+ \i 0) \mu(x)/(\zeta-x).$$
    We consider now the case that the resolvent function $R(z)$ has the values in a Hilbert space of functions. For simplicity we assume  the that the Hilbert space is $\mathfrak{H}= L^2(G)$. In  other papers we treated Hilbert  spaces on subsets of          $\mathbb{R}^n$. \cite{vW}\cite{vW1}.

       Consider the sesqui linear form
       \begin{multline*}f_1,f_2 \in \mathcal{D}(G) \mapsto \langle f_1 |R(\varphi)| f_2 \rangle =\\
       \iiint \d \omega_1 \d \omega_2 \d^2 z \overline{f_1(\omega_1)} f_2(\omega_2)
       \varphi(z) R(z; \omega_1,\omega_2)\end{multline*}
       and assume that $R$ is a distribution in the three variables $\omega_1,\omega_2\in G,z\in G_1\times I$. 
       Then $\mu(x)$ is distribution in the three variables $\omega_1,\omega_2\in G,x\in  G_1$.

      We treat in the following propositiom  a case where the eigen spaces are one dimensional. In \cite{vW,vW1} you will find the case, that the eigen spaces are more dimensional. We will use this proposition for the last example.

        \begin{prop}
.
              Assume that
              there exist two 
             functions $\alpha_x,\alpha_x': x\in G_1\mapsto \mathcal{D}'(G)$, which are $C^\infty$and $\ne 0$ for $x\in G$ and vanish for $ x \notin G$. We assume for  $ f_1,f_2 \in \mathcal{D}(G)$ 
   $$  \langle f_1 | \mu(x)| f_2 \rangle= \langle f_1| \alpha_x \rangle \langle \alpha'_x | f_2 \rangle$$          
   Then for $\zeta \notin G$
      $$ \langle f_1| R(\zeta)| \alpha_x \rangle 
      = \frac{1}{\zeta - x}\langle f_1|  \alpha_x \rangle $$
      or 
    $$    R(\zeta)| \alpha_x \rangle 
      = \frac{1}{\zeta - x} |\alpha_x \rangle $$
      and $| \alpha_x \rangle $ is a right eigenvector of $R(\zeta)$ for the eigenvalue $1/(\zeta -x).$
 Similar
   $$   \langle \alpha_x' | R(\zeta)
      = \frac{1}{\zeta - x} \langle \alpha_x' | $$
 and $  \langle \alpha_x' |$ is a left eigenvector of $R(\zeta)$ for the eigenvalue $1/(\zeta -x).$
       Furthermore for $\varphi \in \mathcal{D}(G)$
      $$ \int  \d x\varphi(x)| \alpha_x \rangle \in L^2(G)$$  
        $$ \int  \d x\varphi(x)\langle \alpha_x |\in L^2(G)$$ 
        So we may  $| \alpha_x \rangle, \langle \alpha_x' |$ consider as generalized vectors in $\mathcal{D}(G)$.
We have the orthogonality relation  for $\varphi_1,\varphi_2 \in \mathcal{D}(G)$
      $$ \iint  \d x \d y\varphi_1(x) \varphi_2(y)\langle \alpha_x'| \alpha_y \rangle 
    = \int\d x \varphi_1(x)\varphi_2(x)$$ 
$$ \langle \alpha_x'| \alpha_y \rangle = \delta(x-y)$$
If $A$
 is a closed operator and  with  the resolvent distribution $R$ and $ AR(z)=R(z)A = -1+zR(z)$ then
$$ A | \alpha_x \rangle = x| \alpha_x \rangle $$
$$  \langle \alpha_x' |A = x \langle \alpha_x' |$$
So $| \alpha_x \rangle$ is a right eigen vector and $ \langle \alpha_x' |$ is a left eigen vector\\ for the eigen value $x$.
 \end{prop}
 
 \begin{pf}
  Then by theorem 1.6 for $\zeta \in( G_1\times I) \setminus G$    
     \begin{multline*} \int \d x \varphi (x) \langle f_1| \mu(x)| R(\zeta)	f_2 \rangle
       =\int \d x \varphi (x) \langle f_1| \mu(x) R(\zeta)	|f_2 \rangle \\
       =   \int \d x \frac{1}{\zeta-x}\varphi (x) \langle f_1| \mu(x)|f_2 \rangle
       \end{multline*}
    As $\varphi$ is arbitrary and    $x \mapsto \langle f_1| \mu(x)|f_2 \rangle$ is continuous we conclude ,
    $$ \langle f_1| \mu(x)R(\zeta)| f_2 \rangle= \langle f_1| R(\zeta)\mu(x)| f_2\rangle
     = \frac{1}{\zeta-x} \langle f_1 |\mu(x)| f_2 \rangle $$
     If $ A$ is an operator such that $AR=RA = -1 +zR$, then $AM=MA = zM$ and we obtain by a similar argument 
   $$ \langle f_1| \mu(x)A| f_2 \rangle= \langle f_1| A\mu(x)| f_2\rangle
     = x \langle f_1 |\mu(x)| f_2 \rangle $$  
  Inserting the special form of $\mu$ we obtain 
     $$\langle f_1| R(\zeta)\mu(x)| f_2\rangle= \langle f_1| R(\zeta)| \alpha_x \rangle 
     \langle \alpha_x' |f_2 \rangle = \frac{1}{\zeta - x}\langle f_1|  \alpha_x \rangle 
     \langle \alpha_x' |f_2 \rangle	   $$
     As $  \langle \alpha_x' |f_2 \rangle$ can be chosen $\ne 0$
     $$ \langle f_1| R(\zeta)| \alpha_x \rangle 
      = \frac{1}{\zeta - x}\langle f_1|  \alpha_x \rangle $$   
      or 
    $$    R(\zeta)| \alpha_x \rangle 
      = \frac{1}{\zeta - x} |\alpha_x \rangle $$
      and $| \alpha_x \rangle $ is a right eigenvector of $R(\zeta)$ for the eigenvalue $1/(\zeta -x).$  The same argument holds for $  \langle \alpha_x' |$.   
 If $f,\varphi \in \mathcal{D}(G)$, then
      $$ \int \d x  \mu(x)| f \rangle\varphi(x) = \mu( \varphi)|f \rangle\in L^2(G)$$
      Furthermore we want to prove that
      $$ \int  \d x\varphi(x)| \alpha_x \rangle \in L^2(G)$$  
          Assume $x\in G$ the there exists an open neighborhood  $x \in N(x)\subset G$ 
      such  there exists a  function   $f \in \mathcal{D}(G)$  with 
       $| \langle  \alpha_x'| f\rangle|\ge 1$
      for $x\in N(x)$ . Assume $ \varphi \in \mathcal{D}(G)$  with support $K$ then there exists a finite family $x_i$ such that
      $ \bigcup_i N(x_i)\supset K$.  For any $i$ there exists a function $f_i$  with  $| \langle  \alpha_x'|f_i\rangle|\ge 1$
      for $x\in N_{x_i}$.  Choose a  partition of unity namely a family
      $\psi_i \in \mathcal{D}(G)$ such that $\sum_i \psi_i =1$ on $K$  and $ \text{support }( \psi_i) \subset N(x_i)$.
      Then
      \begin{multline*} \int  \d x\varphi(x)| \alpha_x \rangle = \sum_i  \int  \d x\varphi(x)\psi_i(x)| \alpha_x \rangle \langle \alpha'_x|	
      f_ i\rangle/\langle \alpha'_x|	
      f_ i\rangle\\
      = \sum _i \int \d x \frac{\varphi(x) \psi_i(x)}{\langle \alpha'_x|	  f_ i\rangle}\mu(x)| f_i \rangle\in L^2(G)
       \end{multline*}
     as  
     $$ \frac{\varphi(x) \psi_i(x)}{\langle \alpha'_x|	  f_ i\rangle}\in \mathcal{D}(G)$$
     Analog we have
       $$ \int  \d x\varphi(x) \langle\alpha_x' | =  \sum_j \frac{\varphi(x) \chi_j(x)}{\langle g_j|\alpha_x \rangle} 
   \langle g_j| \mu(x)\in L^2(G).$$
    Finally
   \begin{multline*}\iint \d x \d y \varphi_1(x) \varphi_2(y) \langle \alpha_x'| \alpha_y \rangle\\
   = \sum_{i,j}\iint \d x \d y\frac{\varphi_1(x) \chi_j(x)}{\langle g_j|\alpha_x \rangle}
   \frac{\varphi_2(y) \psi_i(y)}{\langle \alpha'_y|	  f_ i\rangle}
   \langle g_j| \mu(x) \mu(y)|  f_i \rangle
   \\ =  \sum_{i,j}\int \d x \frac{\varphi_1(x) \chi_j(x)}{\langle g_j|\alpha_x \rangle}
   \frac{\varphi_2(x) \psi_i(x)}{\langle \alpha'_x|	  f_ i\rangle}
   \langle g_j| \mu(x) |  f_i \rangle\\
   =  \sum_{i,j}\int \d x \frac{\varphi_1(x) \chi_j(x)}{\langle g_j|\alpha_x \rangle}
   \frac{\varphi_2(x) \psi_i(x)}{\langle \alpha'_x|	  f_ i\rangle}
   \langle g_j| \alpha_x \rangle  \langle \alpha'_x|  f_i \rangle
   \\=     \sum_{i,j}\int \d x \varphi_1(x) \chi_j(x)
   \varphi_2(x)\psi_i(x)
       = \int \d x \varphi_1(x) \varphi_2(x)\end{multline*}
       using the fact, that $\mu$ is multiplicative.
   \end{pf}

       \begin{rem} We cite Gelfand's definition \cite{Gelfand2}:{\em
         Assume a vector space  $V$ and and a linear mapping $A:V \to V$. A linear functional $F:V \to \mathbb{C}$ is called
         a generalized eigenvector for the eigenvalue $x$ if $F(Af)=xF(f)$ for all $f \in V.$}
     So our eigen vectors are are generalized eigen vectors in Gelfand's sense with $ V = \mathcal{D}(G)$. \end{rem}
      \begin{cor}
      Assume that the resolvent function $R$ has an isolated pole $w$ of finite order outside $G_1 \times I$, then
      in a neighborhood of $w$ 
      $$ R(z) = \sum_k b_k (z-w_1)^{-k} + O(1)$$
       and we have the equations
       $$ b_k| \alpha_x \rangle= 0 ,\quad  \langle \alpha_x' |b_k =0$$
       \end{cor}

       \section{Examples}    
         \subsection{  Finite dimensional matrix}
      By Jordan's normal form one obtains that
               $$ M(z)=\sum _i p_ i \sum_k   (1/k!)(-1)^k a_i^k \partial^k \delta(z- \lambda_i),$$
               where the $\lambda_i$ are the eigenvalues,the $p_i$ are the eigenprojectors,
               $p_ip_j= \delta_{ij}$ and the $a_i$ are nilpotent and $a_ip_j= \delta_{ij}a_i$.   Instead of $\partial$ we could have 
               chosen any other linear combination
                $D$ of $\partial_x$ and $\partial_y$, such that $Dz=1$. This an example, that there are many resolvent distributions extending
                a resolvent function.
               
               We show another example.  Assume $A=0$, then 
               $zR(z)=1$. This equation is solved by $R(z)=1/z-\pi	C \delta(z)$, where $C$ is an arbitrary matrix.  Then  $M(z)= \delta(z)-C \overline{\partial}\delta(z) $ and
               $M(\varphi)M( \psi)=M(\varphi \psi)$  if and only if $C^2=0$.

               \subsection{Unitary operator} Assume that $V$ is a Hilbert space and that $U$ is a unitary operator. Then
               $$ R(z)=\frac{1}{z-U}= \begin{cases}
\displaystyle{\frac{1/z}{1-U/z}}= \sum_{l=0}^\infty U^lz^{-l-1} & \text{ for }|z| > 1\\
\displaystyle{-\frac{1/U}{1-z /U}}= -\sum _{l=0}^\infty U^{-l-1}z^l & \text{ for }|z| < 1
\end{cases}$$
If $\varphi$ is a test function with support in $\mathbb{C} \setminus \{0\}$, then
$$ \int _{||z|-1| > \varepsilon} \d^2 z \varphi(z)R(z)= \int _{|r-1|> \varepsilon } r\d r \int \d 	 \vartheta 
\varphi( r \e^{\i \vartheta}) R( r \e^{\i \vartheta})$$
$$= -\int_{r<1-\varepsilon} r \d r \int \d \vartheta \varphi( r \e^{\i \vartheta})\sum_{l=0}^\infty U^{-l-1} r^l \e^{\i l \vartheta}$$ $$
+\int_{r>1+\varepsilon} r \d r \int \d \vartheta \varphi( r \e^{\i \vartheta})\sum_{l=0}^\infty U^l r^{-l-1} \e^{-\i (l+1)l \vartheta}
$$
Now
$$\int \d \vartheta \varphi( r \e^{\i \vartheta})\e^{\i l \vartheta}
= \frac{1}{1+l^2} \int \d \vartheta \varphi( r \e^{\i \vartheta})(1- \frac{\partial^2}{\partial \theta^2})\e^{\i l \vartheta}$$
$$= \frac{1}{1+l^2} \int \d \vartheta (1- \frac{\partial^2}{\partial \theta^2})\varphi( r \e^{\i \vartheta})\e^{\i l \vartheta}$$
So 
$$ \lim_{\varepsilon \to 0}\int _{||z|-1| > \varepsilon} \d^2 z \varphi(z)R(z)=\int  \d^2 z \varphi(z)R(z) $$
exists. The integral exists clearly for all test functions with support in the open unit circle. So the integral converges  for all test functions or in the sense of distributions. By theorem 1.9 the so defined distribution fulfills the resolvent distribution equation. We calculate
the spectral distribution.
$$ M(\varphi)=-(1/\pi)R(\dc \varphi)= \frac{1}{2 \pi \i }\big(	\int_{|z|= 1+0}R(z)\varphi(z)\d z- 
\int_{|z|= 1-0}R(z)\varphi(z)\d z\big)$$
and obtain
                  $$ \int M(z) \varphi(z) \d^2 z 
                  = \sum_{l= -\infty}^\infty U^l \frac{1}{2 \pi}\int \e^{-\i \vartheta l} \varphi (\e ^{\i \vartheta })
        \d \vartheta. $$
        So
        $$M(\varphi) =\int \d \vartheta \mu (\vartheta) \varphi (\e ^{\i \vartheta }),$$
        where $\mu$ is a distribution on the unit circle given by the Fourier series
        $$ \mu(\vartheta)= \sum_{l= -\infty}^\infty U^l  \e^{-\i \vartheta l} $$
        As $\mu$ is multiplicative one concludes, that $\mu$ is of positive type (s. \cite{Schwartz4}). From there one obtains an easy access to the usual spectral theorem for unitary operators.
        
         \subsection{Selfadjoint operator} Assume $V$ to be a  Hilbert space and $A$ to be a  self adjoint  operator.  Let  $E(x), x \in \mathbb{R}$ be the
        spectral family of $A$,then
        $$ R(\varphi) = \int \d ^2 R(z) \varphi(z) =\iint  \d ^2 z\varphi(z) 1/(z-x) \d E (x)$$
        The integral  can be defined for all $z$  and it is easy to see, that it defines a resolvent distribution. The spectral distribution is given by
        $$ M(x + \i y) = E'(x)\delta(y), $$
              where the derivative is in the sense of distributions.

\subsection{Eigen value problem of  the multiplication operator} The operator $\Omega$ on the real line is defined for any function $f$ by
       $$ (\Omega f) (\omega)= \omega f(\omega)$$
       As operator with a domain in $L^2(\mathbb{R})$ it has the resolvent 
       $$ R(z)= (z-\Omega)^{-1},$$
       defined for $\Im z \ne 0$. The domain of $\Omega$ is $R(z) L^2$ for any $z, \Im z \ne 0$. The 
       resolvent is holomorphic off the real line.       For any test function on the complex plane the integral
      \begin{align*}&
         R(\varphi)= \int d^2 z \varphi(z)/(z-\Omega),&&
           ( R(\varphi)f)(\omega)= \int d^2 z \varphi(z)/(z-\omega)f(\omega)\end{align*}
        exists and we define the resolvent distribution in that way. A short calculation shows, that the resolvent distribution equation is fulfilled.

        For the spectral distribution one has
        $$ M(x+ \i y) = \delta(x + \i y- \Omega) = \delta(x- \Omega) \delta(y) = \mu(x)\delta(y)$$
        $$(M(\varphi)f)(\omega)= (\iint \d x \d y M(x + \i y) \varphi ( x+\i y)f)(\omega)= \varphi(\omega+\i 0)f(\omega)$$ 
                Clearly $M$ is multiplicative. If one extends $\varphi $  to the constant function $1$  to the whole plane (e.g. that
        $\varphi$ converges locally uniformly and stays bounded), then $M(\varphi)$ converges weakly to the identity and  $M$ is complete in this sense.
        
        So 
        $$ \mu(x)= \delta(x-\Omega).$$
        Assume $f_1,f_2 \in \mathcal{D}(\mathbb{R})$ and consider the sesquilinear form
        $$ \langle f_1|\mu(x)| f_2 \rangle  = \int \d x \overline{f_1}(x) f_2(x) 
        = \int \ d x  \langle f_1| \delta_x \rangle \langle \delta_x| f_2 \rangle $$
      and  
      $$ \mu(x) = | \delta_x \rangle \langle \delta_x|,  $$
   where $ | \delta_x \rangle: f \mapsto \langle f | \delta_x \rangle =\overline{f}(x)$ is a right eigen vector of $\Omega$ and
   $\langle \delta_x|: f \mapsto    \langle \delta_x|f \rangle =f(x)$ is a left eigen vector which can be checked easily.  
      We have 
      $$ \int \d x \varphi(x)  | \delta_x \rangle =  | \varphi \rangle  $$
        $$ \int \d x \varphi(x)\langle \delta_x| = \langle \varphi|$$
   Hence $| \delta_x \rangle, \langle \delta_x|$ can be considered as generalized $L^2$ -vectors with the scalar product
    \begin{align*}  \iint \d x_1 \d x_2 (\varphi_1(x_1)\langle \delta_{x_1}|)( \varphi_2(x_2)  | \delta_{x_2} \rangle)&
        = \int \d x \varphi_1(x)\varphi_2(x)\\	
         \langle \delta_{x_1}|\delta_{x_2}\rangle&= \delta(x_1-x_2)\end{align*}

           \subsection{Eigen value problem of the perturbed multiplication operator }     This example is a caricature of the eigenvalue problem arising in the
        theory of radiation transfer
       in a gray atmosphere in plan parallel geometry \cite{EvW}.  We consider for some $c>1$ the set 
       $ G = ]-c,-1[ \cup ]1,c[\subset \mathbb{R}$ and the Hilbert space $\mathfrak{H}= L^2(G)$. On 
       $\mathfrak  {H}$ we define the multiplication operator
       $$ \Omega f(y)=yf(y)$$
       and two real $C^\infty $ functions $g,h$ on $\mathbb{R}$
        with $g(x)>0$ for $ 1<x<c$ and $-c < x< -1$ and 0 outside these two open intervals. We assume 
         $ g(y)=g(-y)$ and 
      $$ h(y)=\begin{cases} -g(y) &\text { for } y>1\\ g(y) & \text{ for } y<-1 \end{cases} $$
       We study $$H= \Omega + |g \rangle \langle h |.$$ 
       The operator is self adjoint in a Krein space. Define
       $Jf(\omega)=f(-\omega),$
       then $JHJ=H^*$ and this is the condition to be self adjoint in the Krein space $\mathfrak{H}$ with the scalar product
       $f_1,f_2 \mapsto \langle f_1, Jf_2 \rangle$.
        We will not use this fact. 
       \begin{prop}
       The operator $H$ is not normal.
       \end{prop} 
       \begin{pf} We have
       $$H^*= \Omega + |h \rangle \langle g |.$$ 
       and
       $$ HH^*= \Omega^2 + \Omega |g \rangle \langle h |+ | g \rangle \langle h | \Omega
       + |g \rangle \langle h | h \rangle \langle g| $$
       $$ H^*H= \Omega^2 + \Omega |h \rangle \langle g|+ | h \rangle \langle g | \Omega
       + |h \rangle \langle g | g \rangle \langle h| $$
       Denote by $S$ the symmetry operator 
      $$(Sf)(\omega)=(1/2(f(\omega)+f(-\omega)).$$
      Then $S \Omega ^2 =\Omega ^2, Sg=g, Sh=0$ and
      $$ S HH^*S = \Omega^2 
       + |g \rangle \langle h | h \rangle \langle g|,\quad \quad SH^*HS= \Omega^2 $$.
       Hence $HH^*\ne H^*H$ and  $H$ is not normal.
       \end{pf}
         Define for $m=0,1,\cdots$ and for any test function $\varphi \in \mathcal{D}(\mathbb{R})$ the norm
 $$ \| \varphi \| _m = \max{\{|\partial_x^k\varphi(x)|:	x\in \mathbb{R},k=1,\cdots,m\}}.$$

       \begin{lem}
Consider the distribution $T_u$ given by the function
$$ T_u(x) = \frac{1 }{x+ \i u}$$
for $u \ne 0$ and by
$$ T_{\pm 0}(x) = \frac{1}{x \pm \i 0}= \frac{\mathcal{P}}{x} \mp \i \pi \delta(x)$$
The distribution $T_u$ is a holomorphic furnction in $ z= x + \i u$ for $u \ne 0$ and is continuous in the lower and in the upper half plane, more precisely $u \mapsto T_u(\varphi)$ is continuous for $ u \ge 0$
and has the boundary value $T_+(\varphi)$ and is continuous for $u \le 0$ and has the boundary value
$T_-(\varphi)$ for any $ \varphi \in \mathcal{D}$. The continuity is uniform for all $\varphi, \| \varphi \|_1 \le 1$ and fixed compact support.
\end{lem}
\begin{pf}
We consider only the convergence $u \to \pm 0$.  We have
$$ T_u(\varphi)= \int \d x \frac{\varphi(x)}{x + \i u)}= \int \d x \frac{x \varphi(x)}{x^2+u^2}
- \i\int \d x \frac{u \varphi(x)}{x^2+u^2}= I+II $$
Then
$$I= - (1/2)\int \d x \partial \varphi(x) \ln (x^2+u^2)\to -\int \d x \partial \varphi(x) \ln|x| = \int \d x\varphi(x) \frac{\mathcal{P}}{x}$$
and
$$ II=- \i \frac{u }{|u|} \int \d t \frac{\varphi( |u| t)}{1+t^2}\to \mp \i \pi \varphi(0)$$
\end{pf}

\begin{lem}
Define the function
$$ R_{\Omega,u}(x) = R_\Omega(x+ \i u)=1/(x + \i u -\Omega).$$
Then $R_u$ converges in the sense of distributions for $u \to \pm 0$ to 
$$R_{\Omega, \pm 0}= \frac{1}{x  \pm\i 0  -\Omega}=\frac{\mathcal{P}}{x- \Omega}\mp \i \pi\delta(x- \Omega)$$
$$ (R_{\Omega,\pm 0}(\varphi)f)(\omega)= \iint \d x \d \omega\frac{\mathcal{P}}{x- \omega}\varphi(x) f(\omega)
\mp \i \pi\iint \d x \d \omega \delta(x- \omega) \varphi(x) f(\omega) $$
The convergence is uniform for all $\varphi$ with $\| \varphi\|_1 \le 1$ with support in a fixed compact set and in operator topology on $\mathfrak{H}=L^2(G)$
\end{lem}
\begin{pf}
 Assume a test function  $\varphi$   and consider for $f_1, f_2 \in L^2(G)$ the expression
$$ \langle f_1 | R_{\Omega,u}(\varphi)| f_2 \rangle= 
\iint \d \omega \d x \frac  {\overline{ f_1(\omega)}f_2(\omega) \varphi(x)}{x+\i u -\omega}$$ $$
=\iint \d \omega \d x \frac  {\overline{ f_1(\omega)}f_2(\omega) \varphi(x+ \omega)}{x+\i u }
= \int \frac{\psi(x)}{x+ \i u}$$
with 
$$ \psi(x)= \int d \omega\overline{ f_1(\omega)}f_2(\omega) \varphi(x+ \omega)$$
and $\psi \in \mathcal{D}$, the support is in a fixed compact interval and 
$$\| \psi \|_1 \le \| \varphi\|_1 \|f_1\| \|f_2\|. $$
Hence the expression converges in operator topology to $R_{\Omega,\pm 0}$
\end{pf}

       The resolvent of $H$  is given by Krein's formula and can be checked easily
       \begin{align*}
       R(z)&=R_\Omega(z)+ R_\Omega(z)|g \rangle \langle h| R_\Omega(z)/C(z)\\
        C(z) &= 1- \langle h| R_\Omega (z)|g \rangle
        = 
       1 - \int_G\d y \frac{g(y)h(y)}{z- y}= 1 + \int _1^c \d y \frac{2 y g(y)^2}{ z^2-y^2}
      \end{align*}
      We discuss at first $C(z)$. The function is well defined and holomorphic outside $G$. Following the lemma and after a variable transform we obtain 
      $$ C(x \pm \i 0)=C_1(x) \pm \i \pi C_2(x)= 1- \int \d y  g(y) h(y)\frac{\mathcal{P}}{x-y} \pm \i\pi  g(x)h(x). $$
      So $ C(x \pm \i 0)\ne 0$
for $z \in G$.     We investigate the zeros of $C(z)$. We have for $z = x + \i u, z \notin G$
      $$ \Im  C(z) = \int _1^c \d y  \frac{4 y^2 xu g(y)^2}{(x^2-u^2 -y^2)^2+ 4 x^2u^2}$$
   
         Hence $C(z)=0$ implies $xu=0$ , so either $x$ or $u$ or both vanish. 
         We have for $x=0$
         \begin{align*}&                                                         
         C(0)=1-\int_1^c \d y \frac{2g(y)^2}{y},&&
       C(\i u)= 1 - \int_1^c \d y \frac{ 2 y g(y)^2}{u^2+y^2}\end{align*}
      So $C(\i u)$ is monotonic increasing for increasing $u^2$ and goes to 1 for $u^2 \to \infty$.
       If $ C(0) < 0$, there exists exactly one $ u_0> 0$
    such that $C(\i u_0)=0$, if $C(0) >0$, then $C(\i u )> 0$ for all $u$.
    
    If $u=0$  and $| x| \le 1$ we have
    $$ C(x)=1-\int_1^c \d y \frac{2yg(y)^2}{y^2-x^2}$$
    and is monotonic decreasing for increasing $x$. If $C(0)>0$ and $C(1)<0$ there exists exactly one $x_0$ with $0<x_0<1$, such that
    $C(x_0)=0$. If $C(0)<0$ there does not exist such an $x$. We do not discuss the case $C(0)> 0$ and $C(1) \ge 0$. In case
    $C(0)=0$ we have a double zero. For $|x| \ge c$ we have $C(x) \ge 1$.

Hence $C(x)$ does not vanish for $|x| \ge c $ in any case.

The singularities of the resolvent  are the slits $[-c,-1]$ and $[1,c]$ and the zeros of $C(z)$. In the neighborhood of a zero   of $C(z)$
we may define a resolvent distribution with the help of proposition 1.10
 We discuss  the behavior of the resolvent in the
neighborhood of 
the slits. 

There exists an open neighborhood  $G_1\subset \mathbb{R}$ of  the closure $\overline{G }$
and  an open interval  $I$ containing 0 in its interior,  such that $C(z) \ne 0$   for $z\in( G_1\times I)\setminus (G\times \{0\})$ 
.
 \begin{prop}
 Consider the restriction of the function $R(z)$ to $( G_1\times I)\setminus ( G\times \{0\}$ and define the distribution 
 $R_u, u\in I \setminus \{0\}$ on  $G_1$ by
 $$ R_u(\varphi)= \int \d x \varphi(x) R(x+ \i u)$$
 with $\varphi \in \mathcal{D}(G_1)$. 
 Then $R_u$ converges in operator norm to $ R_{ \pm \i 0} $ uniformly in $\| \varphi \|_2 \le 1$.
 \end{prop}
\begin{pf}
 Assume a test function  $\varphi$ with support in $G_1$ and consider for $f_1, f_2 \in L^2(G)$ the expression

$$ \langle f_1 | R_u(\varphi)| f_2 \rangle= $$ $$
\iint \d \omega \d x \frac  {\overline{ f_1(\omega)}f_2(\omega) \varphi(x)}{x+\i u -\omega}
+  \iiint \d \omega_1 \d \omega_2 \d x \frac{\overline{f_1(\omega_1)}g(\omega_1)\overline{h(\omega_2)} 
f_2(\omega_2)\varphi(x)}{C(x+ \i u) (x+\i u - \omega_1}. $$

The first term on the  right hand side equals
$$ 
\iint \d \omega \d x \frac  {\overline{ f_1(\omega)}f_2(\omega) \varphi(x)}{x+\i u -\omega}=
\int \d x \frac{\psi(x)}{x + \i u}$$
with
$$ \psi(x)= \int d \omega\overline{ f_1(\omega)}f_2(\omega) \varphi(x+ \omega)$$
and $\psi \in \mathcal{D}$, the support is in a fixed compact interval and 
$$\| \psi \|_1 \le \| \varphi\|_1 \|f_1\| \|f_2\|. $$
Hence the first term converges in operator topology to an operator which we call  $\frac{1}{x +\i 0 - \Omega}$
for $u \to +0$ and  $\frac{1}{x -\i 0 - \Omega}$
for $u \to -0$ .

In the second term we may write
$$\frac{1}{(x +\i u - \omega_1)(x+ \i u - \omega_2) }=
\frac{1}{\omega_2 -\omega_1}
(\frac{1 }{x+ \i u - \omega_1}-\frac{1}{x + \i u - \omega_2})$$
$$= \frac{1}{\omega_2-\omega_1} \int_0^1 \d t
\frac{\d}{\d t} \frac{1}{x + \i u - \omega_1- t(\omega_2-\omega_1)}$$ $$= - \int_0^1 \d t
\frac{\d}{\d x} \frac{1}{x + \i u - \omega_1- t(\omega_2-\omega_1)}.$$ 
Introduce
$$\alpha_u(x)= \varphi(x)/C(x + \i u)$$
then the second term gets
\begin{multline*} - \int_0^1 \d t  \iiint \d \omega_1 \d \omega_2 \d x \overline{f_1(\omega_1)}g(\omega_1)\overline{h(\omega_2)} 
f_2(\omega_2)\alpha_u(x)\\
\times\frac{\d}{\d x} \frac{1}{x + \i u - \omega_1- t(\omega_2-\omega_1)}\end{multline*}
\begin{multline*} =\int_0^1 \d t  \iiint \d \omega_1 \d \omega_2 \d x \overline{f_1(\omega_1)}g(\omega_1)\overline{h(\omega_2)} 
f_2(\omega_2) (\partial_x\alpha_u)(x)\\\times
 \frac{1}{x + \i u - \omega_1- t(\omega_2-\omega_1)}= \int  \d x \frac{\chi_u(x)}{x + \i u}\end{multline*}
 with 
$$ \chi_u(x)=\int_0^1 \d t  \iint \d \omega_1 \d \omega_2 \overline{f_1(\omega_1)}g(\omega_1)\overline{h(\omega_2)} 
f_2(\omega_2) (\partial_x\alpha_u)(x+ \omega_1 + t(\omega_2-\omega1)).$$
Observe that $1/C(x+\i u)$  is a $C^\infty$-  function of $x$ and that it is uniformly $C^\infty$  for
$u \in I$.	Hence
$$ \| \chi_u \|_1 \le \|f_1\| \|f_2\| \|g\| \|h\| \| \alpha_u\|_1 \le \mathrm{const} \|f_1\| \|f_2\| \| \varphi \|_2$$
\end{pf}

 \begin{prop} The resolvent function $R(z)$ can be extended to a distribution on the whole plane,
 which fulfills the distribution resolvent equation. 
 \end{prop}
 \begin{pf}
The resolvent function $ R(z)=(z-H)^{-1}$ is defined and holomorphic  outside the spectrum, which consists of $G$ and the zeros of 
$C(z)$, where $R(z)$ has simple or double poles. Outside the spectrum the resolvent function clearly is a distribution,
in the neighborhood of the spectrum the resolvent function can be extended to a distribution by subsections 1.7 and 1.8. So $R$ can be extended to a distribution on the whole plane
and as the spectrum has Lebesgue measure zero, the distribution $R$ fulfills the distribution resolvent equation.      
\end{pf}
 \begin{prop}
We calculate the spectral distribution. We write $ \delta_2$ for the two dimensional  $\delta$- function in order to dstinquish	from the one dimensional one. If there are two zeros $\ne 0$ we obtain
$$ M(z)=M(x+\i u)=r_+\delta_2(z-z_0)+r_-\delta_2(z+z_0)+\delta(u)\mu(x)$$
where $r_\pm$ are the residues of $R(z)$ at the points $\pm z_0$ .
In the case of a double zero at $z=0$ 
we expand 
$$ R(z) =  z^{-2}a+ z^{-1}p_0+ \cdots$$
and
$$M(z)=M(x + \i u)=  p\delta_2(z) - a \partial  \delta_2(z) + \mu(x)\delta(u),$$ 
In both cases
$$ \mu(x)= \frac{1}{2\pi \i}(R(x-\i 0)-R(x+\i 0)).$$
We consider the sesquilinear form 
  $$ f_1,f_2 \in \mathcal{D}(G)\to \langle f_1| R(z)| f_2 \rangle. $$
  It  is a distribution in three variables , the variables $\omega_1,\omega_2$ of the functions $f_1,f_2$ and $z$. 
  In the case of two zeros
$$ M(z) = |\alpha_+\rangle \langle\alpha_+'|\delta(z-z_0)+ |\alpha_- \rangle \langle \alpha_-'|\delta(z+z_0)+
 \delta(u) |\alpha_x \rangle \langle \alpha_x' |$$
 with the right resp. left usual eigenvectors
 \begin{align*}
&| \alpha_\pm \rangle = \frac{1}{\sqrt{ \langle h|(\pm z_0-\Omega)^{-2}|g \rangle}}\frac{1}{ \pm z_0-\Omega}| g \rangle\\&
 \langle\alpha_\pm' |= \frac{1}{\sqrt{ \langle h|(\pm z_0-\Omega)^{-2}|g \rangle}}\langle h |\frac{1}{ \pm z_0-\Omega}
 \end{align*}
 and for $x \in G$ the right, resp. left generalized eigenvectors
 \begin{align*}&
 | \alpha_x \rangle = \frac{1}{\sqrt{C_1^2 + \pi^2C_2^2}}(C_1(x)| \delta_x \rangle + h(x)A(x))\\&
  \langle\alpha_x' |= \frac{1}{\sqrt{C_1^2 + \pi^2C_2^2}}(C_1(x) \langle\delta_x |  + g(x)A'(x))
  \end{align*}
  with
  \begin{align*}&
   C_1(x)=1-\int \d y g(y)h(y)\frac{\mathcal{P}}{x-y}&&
  C_2(x)=g(x)h(x)\\&
  A(x)=\frac{\mathcal{P}}{x-\Omega}| g \rangle &&
  A'(x)= \langle h |\frac{\mathcal{P}}{x-\Omega} \\&
  B(x) = g(x) | \delta_ x \rangle&&
  B'(x)=h(x) \langle \delta_x |\end{align*}
 In the case $C(0)=0$ we obtain
 $$M(z)=M(x + \i u)= \ p \delta_2(z) - a \partial  \delta_2(z) + \delta(u)| \alpha_x 
\rangle \langle\alpha'_x |,$$
where $\alpha_x \alpha_x'$ are given by the formulas above
and
$$ a=\frac{ \Omega^{-1}| g \rangle \langle h | \Omega^{-1}}{\langle h| \Omega^{-3}|g \rangle}$$
$$ p = \frac{\Omega^{-2}| g \rangle \langle h | \Omega^{-1}+
\Omega^{-1}| g \rangle \langle h | \Omega^{-2}
}{\langle h| \Omega^{-3}|g \rangle}$$
\\

Here $p^2=p$ and $a^2=0$ and $ap=pa=a$. 
We obtain for $\vartheta,\vartheta' = \pm$ and $x,x' \in \mathbb{R}$ the orthogonality relations
\begin{align*}&
 \langle \alpha_\vartheta' | \alpha_{\vartheta'}\rangle = \delta_{\vartheta,\vartheta'}&&
 \langle \alpha_x'| \alpha_\vartheta \rangle =\langle \alpha_\vartheta | \alpha_ x\rangle = 0 &&
\langle \alpha_x| \alpha_{x'}\rangle = \delta(x-x')\end{align*}
Analogous relations hold for the case $C(0)=0$. The spectral distribution is complete, i.e.

$$M(1)= \int d^2 z M(z)= \sum_{\vartheta = \pm}| \alpha_\vartheta ' \rangle \langle \alpha _\vartheta |
+ \int_G \d x | \alpha_x ' \rangle \langle \alpha _x |= 1$$
resp.
$$M(1)= \int d^2 z M(z)= p
+ \int_G \d x | \alpha_x ' \rangle \langle \alpha _x |= 1$$
\end{prop}

\begin{pf}
We recall 
      \begin{align*}
       R(z)&=R_\Omega(z)+ R_\Omega(z)|g \rangle \langle h| R_\Omega(z)/C(z)\\
        C(z) &= 1- \langle h| R_\Omega (z)|g \rangle
        = 
       1 - \int_G\d y \frac{g(y)h(y)}{z- y}= 1 + \int _1^c \d y \frac{2 y g(y)^2}{ z^2-y^2}
      \end{align*}
Assume at first, that $C(z)$ has two zeros $\pm z_0= \pm \i u_0$ or $ \pm z_0= \pm x_0$.
.
 The residuum of
the complex function $R(z)$ at the points $ \pm z_0$ is given by
$$r_\pm= \big(\frac{1}{\pm z_0 - \Omega}|g \rangle \langle h | \frac{1}{\pm z_0-\Omega}\big)/
C'(\pm z_0)=$$ $$ \big(\frac{1}{\pm z_0 - \Omega}|g \rangle \langle h | \frac{1}{\pm z_0-\Omega}\big)
\frac{1}
{\langle h| (\pm z_0-\Omega)^{-2}|g \rangle}$$

Using the results of  proposition 1.10 we obtain
$$ M(z) = M(x+\i u)= r_+\delta_2(z-z_0)+ r_-\delta_2(z+z_0)+ \delta(u)\mu(x)$$
where 
      $$\mu(x)= \frac{1}{2 \pi \i}(R(x-\i 0)-R(x+ \i 0)$$
   We  consider the case $C(0)=0$ .We expand at the origin
$$ C(z)= 1+ \sum_{n=0}^\infty z^n \langle h| \Omega ^{-(n+1)}|g \rangle =
z^2 \langle h| \Omega^{-3}|g \rangle+ O(z^3)=$$ $$
-z^2 \int _1^\infty 2g(y)/y^3 \d y + O(z^3)  $$

$$R(z)=z^{-2}\frac{ \Omega^{-1}| g \rangle \langle h | \Omega^{-1}}{\langle h| \Omega^{-3}|g \rangle}+
 z^{-1} \frac{\Omega^{-2}| g \rangle \langle h | \Omega^{-1}+
\Omega^{-1}| g \rangle \langle h | \Omega^{-2}
}{\langle h| \Omega^{-3}|g \rangle}+ O(1)$$ $$= z^{-2}a+ z^{-1}p_0+ O(1)$$
and
$$M(z)=M(x + \i u)= \ p_0 \delta_2(z) - a \partial  \delta_2(z) + \mu(x)\delta(u).$$  
Here $p^2=p$ and $a^2=0$ and $ap=pa=a$ and $\mu(x)$ given by $R(x \pm \i 0)$ as above. We have
$$ C(x \pm \i 0)= 
1-  \langle h| \frac {\mathcal{P}}{x-\Omega}| g \rangle 
\pm \i \pi \langle h| \delta(x- \Omega) |g \rangle = C_1(x) \pm \i \pi C_2(x) $$

    $$   R(x \pm \i 0)=
    R_\Omega(x \pm \i 0)+ R_\Omega(x \pm \i 0)|g \rangle \langle h| R_\Omega(x \pm \i 0)/C(x \pm \i 0)$$
$$= \frac{\mathcal{P}}{x-\Omega}\mp \i \pi \delta(x- \Omega) + ( A \mp \i \pi B)(A' \mp \i \pi B') / (C_1\pm \i  \pi C_2) $$
 Use the relation  $\delta(x-\Omega)=|\delta_x \rangle \langle \delta_x|$. and obtain
\begin{multline*} 
     \mu(  x)=\frac{1}{2\pi \i}
   (R(x- \i 0)-R(x+ \i 0) )\\
      = \delta(x-\Omega)+ \frac{1}{C_1^2+\pi^2 C_2^2}
       (AC_2A'+AC_1B'+BC_1A'- \pi^2 BC_2B')   \\ 
       =\frac{1}{C_1^2+\pi^2 C_2^2}\big((C_1^2+\pi^2 C_2^2)|\delta_x \rangle \langle \delta_x|+  AC_2A'+AC_1B'+BC_1A'- \pi^2 BC_2B'\big)\\
 = \frac{1}{C_1^2+\pi^2 C_2^2}\big(Ah(x)g(x)A'+AC_1h(x) \langle \delta_x|
+ g(x)|\delta_x \rangle C_1 A' + C_1^2 | \delta_x \rangle \langle \delta_x |\big)\\=
|\alpha_x \rangle \langle \alpha'_x |
 \end{multline*}

That $M$ is complete follows from theorem 1.8..
The  orthogonality relations follow from subsections 1.7 and 1.8.  Obviously  
$x \mapsto\langle	f|\alpha_x 	\rangle \rangle,\langle\alpha_x'|f\rangle \rangle$
 is $C^\infty$ for $f \in \mathcal{D}(G)$. 
To prove that $\alpha_x \ne 0$ choose
a $\psi$  with $ \psi	(0)=1$ and support in $[-\varepsilon, \varepsilon]$ and $\varepsilon$ sufficiently small and put 
$ f(y) =(x-y)\psi(x-y)$.

  \end{pf}
 
 \begin{rem}
 The orthogonality relation $\langle \alpha_x| \alpha_{x'}\rangle = \delta(x-x')$  can be easily checked by hand
 using lemma 1.1.
 \end{rem}

\begin{pf}
$$\langle \alpha_x| \alpha_{x'}\rangle= \frac{1}{\sqrt{C_1(x)^2 + \pi^2C_2(x)^2}}
\frac{1}{\sqrt{C_1(y)^2 + \pi^2C_2(y)^2}}\int \d \omega T(x,y,\omega)$$
with
$$ T(x,y,\omega)=(C_1(x)\delta(x-\omega)+g(x)h(\omega)\frac{\mathcal{P}}{x-\omega})
(C_1(y)\delta(y-\omega)+g(\omega)h(y)\frac{\mathcal{P}}{y-\omega})$$
$$=C_1(x)C_1(y)\delta(x-\omega)\delta(y-\omega	)$$ $$+C_1(x)\delta(x-\omega)g(\omega)h(y)\frac{\mathcal{P}}{y-\omega}
+g(x)h(\omega)\frac{\mathcal{P}}{x-\omega}C_1(y)\delta(y-\omega)$$
$$+g(x)h(y)g(\omega)h(\omega)\bigg(\frac{\mathcal{P}}{y-x}\left(\frac{\mathcal{P}}{x-\omega} -\frac{\mathcal{P}}{y-\omega}
\right) + \pi^2 \delta(x-\omega)\delta(y-\omega)\bigg) $$
using lemma 1.1. Then
$$ \int \d \omega T(x,y,\omega)=C_1(x)^2 \delta(x-y)+C_1(x)g(x)h(y)\frac{\mathcal{P}}{y-x}
+C_1(y)g(x)h(y)\frac{\mathcal{P}}{x-y}$$										
$$+g(x)h(y)\frac{\mathcal{P}}{y-x}\big(1-C_1(x)-1+C_1(y)\big)+\pi^2 g(x)^2h(x)^2\delta(x-y)$$ $$
= (C_1(x)^2+\pi^2C_2(x)^2)\delta(x-y)$$
From there follows the orthogonality relation.
\end{pf}

\end{document}